\documentclass[a4paper,twoside]{article}
\usepackage{amsmath,amssymb,amsthm}
\usepackage{amscd}
\usepackage[usenames,dvipsnames]{color}
\usepackage{pstricks}
\usepackage{graphicx}
\usepackage[all]{xy}
\usepackage{tabularx}

\newtheorem{theorem}{Theorem}[section]
\newtheorem{lemm}[theorem]{Lemma}

\theoremstyle{definition}
\newtheorem{defi}[theorem]{Definition}
\newtheorem{prop}[theorem]{Proposition}

\newtheorem{ex}[theorem]{Example}

\theoremstyle{remark}

\numberwithin{equation}{section}

\begin{document}

\title{The Euler characteristic of a surface
from its Fourier analysis in one direction.}


\author{Nguyen Viet Dang\\
Laboratoire Paul Painlev\'e (U.M.R. CNRS 8524)\\
UFR de Math\'ematiques\\
Universit\'e de Lille 1,
France\\}


\date{}

\maketitle

\begin{abstract}
In this paper, we prove that 
we can recover the genus of a
closed 
compact surface $S$ in
$\mathbb{R}^3$
from the restriction
to a generic line
of the Fourier transform
of the canonical 
measure 
carried by $S$.
We also show 
that the restriction
on some line 
in Minkowski space
of the solution
of a linear wave equation
whose Cauchy data
comes from 
the canonical 
measure 
carried by $S$,
allows
to recover
the Euler
characteristic of $S$.
\end{abstract}

\section*{Introduction}

Let us start with a surface $S$ 
in $\mathbb{R}^3$. To this surface $S$,
we can always associate
a natural measure
just by 
integrating 
test functions on $S$.
Indeed, 
the surface carried measure $\mu$
is the distribution defined as
\begin{equation}
\mu(\varphi)=\int_S \varphi d\sigma 
\end{equation}
where $d\sigma$ is the canonical area element 
on $S$ induced by the 
Euclidean metric of $\mathbb{R}^3$
(\cite[p.~334]{SteinShakarchi},
\cite[p.~321]{SteinBeijing})
and $\varphi$ a test function.
If we assume
that $S$ is compact,
then we find that
the measure $\mu$
can be viewed as a 
\emph{compactly
supported
distribution}
on $\mathbb{R}^3$,
$\mu$ is thus a tempered distribution 
and therefore it 
has a well defined
Fourier transform denoted by 
$\widehat{\mu}$.

In harmonic analysis,
we are interested in the
analytical properties
of the Fourier transform $\widehat{\mu}$
of measures $\mu$ which are
supported on submanifolds
of some
given vector
space. 
For instance, in 
a very recent paper
\cite{Aizenbud-12},
using
the concept of wave front set
and resolution of singularities, 
Aizenbud and Drinfeld
give a proof
that the Fourier transform
of an algebraic measure
is smooth in some 
open dense set.
More classically, one
would like to study the 
asymptotic behaviour of
$\widehat{\mu}$ for large
momenta. In the simple case
where $\mu$
is the measure carried by a plane $H$ in 
a vector space $V$, 
the Fourier transform $\widehat{\mu}$ is a 
measure carried by the dual plane $H^\perp$
in Fourier space $V^*$ \cite[Thm 7.1.25 p.~173]{HormanderI}. 
Moreover, if the measure 
$\mu$ 
is compactly supported, 
then we know
by Paley--Wiener
that its 
Fourier transform $\widehat{\mu}$
should be a 
\emph{bounded} function 
on
$\mathbb{R}^3$. Therefore
in both cases, for any 
test function $\varphi$,
classical theorems
in distribution 
theory only
yield
that the Fourier
transform
$\widehat{\mu\varphi}$ 
is \textbf{bounded}.

However, if we assume that 
$\mu$ is carried by a \emph{compact} 
hypersurface $S$ of dimension $n$ 
with \emph{non vanishing Gauss curvature},
then a celebrated result of Stein
(see \cite[Theorem 1 p.~322]{SteinBeijing} and also \cite[Thm 7.7.14]{HormanderI})
gives finer decay properties
on $\widehat{\mu}$. Indeed, he proves that
$\vert \widehat{\mu}\vert \leqslant C (1+\vert\xi\vert)^{-\frac{n}{2}}$
for some constant $C$ which depends 
on the volume of $S$ 
and the Gauss curvature.
Stein's result shows 
the 
interplay 
between geometry and analysis, since
a simple assumption on 
the Gauss curvature of
$S$ gives
sharper decay properties on 
$\widehat{\mu}$
than a simple application
of the Paley--Wiener theorem.

In this paper,
we explore
the relationship
between the Fourier transform
$\widehat{\mu}$
of the surface 
carried measure $\mu$ 
and the topology
of the surface itself.
In the same spirit as
in the papers \cite{KDB, Kashiwaraindex}, 
we use microlocal analysis and Morse theory 
in the study of
$\widehat{\mu}$.
Throughout the paper,
the surface $S$ 
will always be assumed
to be smooth and oriented.
We state in an informal way
the main result of this note:
\begin{theorem}\label{mainthm}
Let $S$ be a closed
compact 
surface 
embedded in $\mathbb{R}^3$
and $\mu$ the associated 
surface carried measure.
For a
generic line
$\ell\subset \mathbb{R}^3$,
we can recover
the Euler characteristic 
of the surface $S$
from the restriction 
$\widehat{\mu}|_\ell$.
\end{theorem}
The space of unoriented 
lines in $\mathbb{R}^3$
is canonically identified
with the projective
space $\mathbb{RP}^2$
and
generic means that
the theorem
holds true
for
an open dense set of
lines in $\mathbb{RP}^2$.
We refer the reader to
Theorem \ref{mainthmprecise} which 
explains in \emph{what sense we 
recover} $\chi(S)$ and
gives a precised version
of our main result.

Our main result 
might look surprising 
because 
if we knew the
full Fourier transform
$\widehat{\mu}$, then 
it would be easy to reconstruct
$\mu$ hence $S$ from $\widehat{\mu}$ 
by Fourier inversion. 
But to 
recover the topology of $S$, it
is enough
to consider the
partial information
of the restriction of $\widehat{\mu}$
to a line $\ell$.
Actually,
all the information
we need is contained
in the asymptotic 
expansion of
$\widehat{\mu}(\xi)$
for large $\vert\xi\vert$.
In a way,
our result 
is reminiscent
of Weyl's tube formula
where 
the asymptotic expansion
in $\varepsilon$
of the volume of the tube
$S_\varepsilon$
of ``thickness'' $\varepsilon$
around a surface $S$
gives geometrical
data on $S$:
the volume
and the Euler 
characteristic of $S$.

As a byproduct of our main theorem, we 
derive similar results 
for more general integral 
transforms.
First, we connect
our result with the
Radon transform
in the spirit of \cite[section 5.3]{WF1}.
We denote by $\mathcal{R}\mu$
the Radon transform of $\mu$.
\begin{theorem}\label{radon}
Let $S$ be a closed
compact 
surface 
embedded in $\mathbb{R}^3$
and $\mu$ the associated 
surface carried measure.
For a
generic line
$\ell\subset \mathbb{R}^3$,
we can recover
the Euler characteristic 
of the surface $S$
from the restriction of 
$\mathcal{R}\mu$ on $\ell$.
\end{theorem}
In the sequel,
a Morse function
$\psi$
is called 
\emph{excellent}
if its critical values are
distinct.
\begin{theorem}\label{thm2}
Let $S$ be a closed
compact 
surface embedded in
$\mathbb{R}^3$
and $\mu$ the associated 
surface carried measure.
If $\psi\in C^\infty(\mathbb{R}^3)$ is such that
its restriction on $S$ is Morse excellent then 
we can recover
the Euler characteristic of $S$
from the map $\lambda\longmapsto\mu\left(e^{i\lambda \psi} \right)$.
\end{theorem}

And finally, 
let us 
state informally 
an application
of this theorem
to the following 
inverse problem:
\begin{theorem}
\label{thm3}
Let $S$ be a closed
compact 
surface embedded in
$\mathbb{R}^3$ and
$\mu$ the associated 
surface carried measure.
Consider the solution 
$u\in\mathcal{D}^\prime(\mathbb{R}^{3+1})$
of the wave equation 
$\left(\partial^2_t-\sum_{i=1}^3\partial_{x^i}^2\right) u=0$ 
with Cauchy data $u(0)=0,\partial_tu(0)=\mu$.
For $x$ in some
open dense subset of $\mathbb{R}^3$,
set $\ell(x)$ to be
the line $\{x\}\times \mathbb{R}\subset \mathbb{R}^{3+1} $,
then: 
\begin{itemize}
\item the restriction
$u_{\ell(x)}$ is a compactly 
supported distribution of $t$, 
\item we can recover the 
Euler characteristic of
$S$ from the restriction 
$u_{\ell(x)}$.
\end{itemize}
\end{theorem}
We give precise statements in Proposition
\ref{claim1thm3} and Theorem
\ref{thm3precise} which explains
in what sense we recover $\chi(S)$.

In appendix,
we gather several useful
results in Morse theory
which are used in the
paper.

\subsubsection{The general principle underlying our results.}
The main idea of our note 
is to think of the stationary phase 
principle
as an analytic 
version of Lagrangian intersection.
The first observation is that a surface carried
measure $\mu$ is the simplest example of 
Lagrangian distribution
(in the sense
of Maslov H\"ormander)
with
wave front set 
the conormal
$N^*(S)$
of the surface $S$.
To probe the wave front
set of a distribution, 
a natural idea is to
calculate 
its Fourier transform
or more generally to
pair $\mu$
with an oscillatory
function of the 
form $e^{i\lambda \psi}$
where $d\psi\neq 0$.
We can think
of this operation
as
a 
plane wave analysis
of $\mu$.
By the stationary phase
principle,
the main contributions
to
the asymptotics of
$\mu(e^{i\lambda \psi})$
when $\lambda\rightarrow +\infty$
come from
the points where
the graph of $d\psi$
meets the wave front set
of $\mu$
which is the conormal
$N^*(S)$.
Therefore our strategy is to
extract topological information
from the Lagrangian intersection
$\left(\text{Graph of }d\psi\right)
\cap N^*(S)$
by the stationary phase principle.
This is strongly
related to the
work
of Kashiwara \cite[p.~194]{Kashiwaraindex}
and Fu \cite{Fuindex}
where the Euler
characteristic of
subanalytic sets 
(Kashiwara gives an index
formula in the context of
\emph{constructible sheaves})
is expressed
in terms of Lagrangian
intersection.

Let us give the central example
of the theory
\begin{ex}
Let $\psi$ be a Morse function
on a manifold
$M$ of dimension $n$
and $\Omega$ a test form in 
$\Omega^n_c(M)$, 
then
the main contributions
to the oscillatory
integral 
$\int_M e^{i\lambda \psi}\Omega$
when 
$\lambda\rightarrow +\infty$
come from the critical points
of $\psi$, 
in other words,
from the intersection of
the Lagrangian graph $d\psi$
and the zero section 
$\underline{0}\subset T^* M$.
\end{ex}

\section{Proof of Theorem \ref{mainthm}.}
\subsection{Fourier transform, Morse functions 
and stationary phase.}
Recall that $S$ is a 
closed,
compact 
surface embedded in
$\mathbb{R}^3$
and $\mu$ denotes 
the associated 
surface carried measure.
We denote by $d\sigma$ 
the canonical area element
on $S$ induced by 
the Euclidean metric 
of
$\mathbb{R}^3$.
Then
the surface carried measure $\mu$
is the distribution 
defined by the formula:
\begin{eqnarray}
\forall\varphi\in\mathcal{D}(\mathbb{R}^3), 
\mu(\varphi)=\int_S \varphi  d\sigma.
\end{eqnarray}
The Fourier transform $\widehat{\mu}$ reads
\begin{equation}\label{oscint}
\widehat{\mu}(\xi)=\int_S e^{-i\xi(x)} d\sigma(x).
\end{equation}
Since $S$ is compact, $\mu$ is compactly supported
thus $\widehat{\mu}$ is a
real analytic function 
by 
Paley--Wiener
theorem.
In order to study
the asymptotics of 
$\widehat{\mu}(\lambda\xi)$
for $\lambda$ large,
we will think
of
$\xi\in\mathbb{R}^{3*}\subset C^\infty(\mathbb{R}^3)$
as a
linear function 
on
$\mathbb{R}^3$ 
whose
restriction
on $S$
is
a height function 
which we
denote by $\xi\in C^\infty(S)$.

For every
$\xi\in\mathbb{R}^{3*}$, we denote
by $[\xi]$ its class in $\mathbb{RP}^2$ and
by Theorem 
\ref{regvalGaussmorse} recalled 
in Appendix, 
for an everywhere dense set
of $[\xi]\in\mathbb{RP}^2$, 
the function $\xi$ is a 
\emph{Morse function}
on $S$.
Set $\omega=\frac{\xi}{\vert\xi\vert}$, 
$\lambda=\vert\xi\vert$ and note that
$\omega$ has 
the same critical points
as $\xi$.
The stationary phase principle 
states that
the main
contributions to 
the asymptotics of
$\widehat{\mu}(\lambda\omega)$
when $\lambda\rightarrow\infty$
come from the 
\emph{critical points} of the Morse
function $\omega$ which
are isolated
by Theorem 10.4.3 in
\cite[ p.~87]{DubrovinFomenkoNovikovII}.
We denote by $\text{Crit }\omega$
the
set of 
critical points of $\omega$.
Then in the neighborhood
of every $x\in \text{Crit }\omega$, 
we choose
local coordinates
$(y^1,y^2)$
on $S$ 
such that
$y^1(x)=y^2(x)=0$ and
$d\sigma(x)=\vert dy^1dy^2\vert$
where $d\sigma$ is the canonical
area element on $S$.
The method
of stationary phase (\cite[Lemma (19.4) p.95]{Eskin},\cite[p.~124]{BatesWeinstein}) yields
the asymptotic expansion:
\begin{equation}\label{statphasemu}
\widehat{\mu}(\lambda\omega)\underset{\lambda\rightarrow \infty}{\sim} \sum_{x\in \text{Crit }\omega}\left(\frac{2\pi}{\lambda}\right)^{\frac{n}{2}}e^{i\frac{\pi}{4}(n_+-n_-)(x)} 
\frac{e^{-i\lambda\omega(x)}}{ \sqrt{\vert\det \omega^{\prime\prime}(x) }\vert }\left[ 1+\sum_{j=1}^\infty b_j(x)\lambda^{-j} \right].
\end{equation}

We want to comment on the geometric
interpretation of each of the terms
in the expansion:
\begin{itemize}
\item $\frac{n}{2}=1$ in our case since $S$ is a surface.
\item $n_+(x)$ (resp $n_-(x)$) is the number of positive 
(resp negative) eigenvalues of
the Hessian $-\omega^{\prime\prime}$ 
of the Morse function $-\omega$
at the critical point $x$, 
the number $n_-(x)$ is 
the \emph{Morse index} of
$-\omega$ at $x$.
For surfaces, 
$e^{i\frac{\pi}{4}(n_+-n_-)(x)}$ 
can only take 
the three values $\{i,1,-i\}$.
Observe that 
\begin{eqnarray*}
n_-(x)=1\mod(2)&\Leftrightarrow & e^{i\frac{\pi}{4}(n_+-n_-)(x)}=1\\ 
n_-(x)=0\mod(2) &\Leftrightarrow &e^{i\frac{\pi}{4}(n_+-n_-)(x)}=\pm i.
\end{eqnarray*}
\item In the local chart $(y_1,y_2)$
around $x$, 
the Hessian $-\omega^{\prime\prime}(x)$ 
is non degenerate
since $\omega$ is Morse
and it also
coincides with the
second fundamental form of 
the surface $S$ at $x$. 
Therefore 
$\vert K(x)\vert
=\vert\det\omega^{\prime\prime}(x)\vert$ 
where $K(x)$ 
is the \emph{Gauss curvature} of
the surface $S$ at $x$ (\cite[3.3.1 p.~67]{NovikovTaimanov}).
\end{itemize}

\subsection{An oscillatory integral
whose singular points
are the critical values of the height function.}

We denote by 
\begin{eqnarray*}
\mathcal{F}_\lambda^{-1}:v\in\mathcal{S}^\prime(\mathbb{R})\longmapsto \widehat{v}(\tau)=\frac{1}{2\pi}\int_\mathbb{R} d\lambda e^{i\tau\lambda}v(\lambda)
\end{eqnarray*}
the
inverse Fourier transform w.r.t. variable
$\lambda$.
We define the oscillatory integral
$u=\mathcal{F}_\lambda^{-1}
\left(\left(\frac{\lambda}{2\pi}\right) 
\widehat{\mu}(\lambda\omega)\right)$ on $\mathbb{R}$
and show
that $u$ is singular at the critical values
of the Morse function $\omega$.
Recall that
$K(x)$
denotes 
the Gauss curvature of $S$ at $x$. 
\begin{prop}\label{Lagdistrib}
Let $S$ be a closed
compact
surface embedded in
$\mathbb{R}^3$
and $\mu$ the associated 
surface carried measure.
Let $u$ be the distribution 
\begin{eqnarray}\label{uoscill}
u=\mathcal{F}_\lambda^{-1}
\left(\left(\frac{\lambda}{2\pi}\right) 
\widehat{\mu}(\lambda\omega)\right).
\end{eqnarray}
 
 If
$\omega\in C^\infty(S)$ is Morse,
then:
\begin{itemize}
\item the singular support 
of $u$
is the set of critical values
of $\omega$,
\item $u$ is a finite 
sum of oscillatory integrals
on $\mathbb{R}$
\item each
oscillatory integral has polyhomogeneous 
symbol whose leading term is
$\frac{e^{i\frac{\pi}{4}(n_+-n_-)(x)}}{\sqrt{\vert K(x)\vert}}$ 
where $x\in\text{Crit }\omega$.
\end{itemize}
\end{prop}
\begin{proof}
By the stationary
phase expansion,
\begin{eqnarray*}
\left(\frac{\lambda}{2\pi}\right) 
\widehat{\mu}(\lambda\omega)=\sum_{x\in \text{Crit }\omega}
e^{-i\lambda \omega(x)}b(x,\lambda)
\end{eqnarray*}
where, for every $x\in\text{Crit }\omega$,
$b(x,\lambda)$ is
a polyhomogeneous symbol
in $\lambda$:
\begin{eqnarray}\label{asymptb}
b(x;\lambda)\sim \frac{e^{i\frac{\pi}{4}(n_+-n_-)(x)}}{\sqrt{\vert\det \omega^{\prime\prime}(x)\vert}}\left[ 1+\sum_{j=1}^\infty b_j(x)\lambda^{-j} \right] 
\end{eqnarray}
with 
principal symbol
$ \frac{e^{i\frac{\pi}{4}(n_+-n_-)(x)}}{\sqrt{\vert\det \omega^{\prime\prime}(x)\vert}}$
since $\vert K(x)\vert=\vert\det \omega^{\prime\prime}(x)\vert\neq 0$.
Therefore:
\begin{eqnarray*}\label{formulau}
u(t)&=& \mathcal{F}_\lambda^{-1}
\left(\left(\frac{\lambda}{2\pi}\right) 
\widehat{\mu}(\lambda\omega)\right)\\
&=&\sum_{x\in \text{Crit }\omega}\frac{1}{2\pi}\int_{-\infty}^{+\infty}
e^{i\lambda(t-\omega(x))}b(x;\lambda) d\lambda 
\end{eqnarray*}
is a finite sum 
of oscillatory integrals
(\cite[Theorem (7.8.2) p.~237]{HormanderI}).
By the 
general theory 
of oscillatory 
integrals
\cite[Theorem IX.47 p.~102]{ReedSimonII}
\cite[Theorem 8.1.9 p.~260]{HormanderI}:
\begin{eqnarray}
WF(u)=\{(t;\tau) | (\omega(x)-t)=0, \tau\neq 0  \}.
\end{eqnarray} 
From 
\cite[Proposition 8.1.3 p.~254]{HormanderI},
we deduce that
the singular support of 
$u$
is
$\{t| (\omega(x)-t)=0, x\in \text{Crit }\omega \}$ 
which is exactly
the set of critical values of
$\omega$.
\end{proof}
We state and prove a precised version of the main theorem \ref{mainthm}
given in the introduction:
\begin{theorem}\label{mainthmprecise}
Let $S$ be a closed
compact 
surface embedded in 
$\mathbb{R}^3$
and $\mu$ the associated 
surface carried measure. 
Then for an everywhere dense set $\omega\in\mathbb{S}^2$, 
the distribution
$u=\mathcal{F}_\lambda^{-1}
\left(\left(\frac{\lambda}{2\pi}\right) 
\widehat{\mu}(\lambda\omega)\right)$
can be canonically 
decomposed as a sum 
\begin{equation}\label{decompdirac}
u=\sum_{x\in \text{Crit }\omega} a(x)\delta_{\omega(x)} + r
\end{equation}
such that
\begin{itemize}
\item $\forall x\in \text{Crit }\omega, a(x)\neq 0$
\item $r$ is an oscillatory integral
with symbol
of degree $-1$. 
\end{itemize}
The Euler characteristic
of $S$ satisfies the identity:
\begin{equation}\label{Morsepol}
\chi(S)=\sum_{x\in \text{Crit }\omega}-\frac{a(x)^2}{\vert a(x)\vert^2}. 
\end{equation}
\end{theorem}
\begin{proof}
By Lemma \ref{regvalGaussmorse2} 
proved in appendix, 
for an everywhere dense set $\omega\in\mathbb{S}^2$,
the corresponding function 
$\omega\in C^\infty(S)$ is
an excellent Morse function which means
that if $(x_1,x_2)$ are distinct critical points
of $\omega$ then
$\omega(x_1)\neq \omega(x_2)$.

In particular, $\omega$ is Morse
therefore we can use
the results of Proposition \ref{Lagdistrib}.
From the asymptotic expansion (\ref{asymptb}) of the symbol $b$ 
and the equation of $u$ (\ref{formulau}), we find that: 
\begin{eqnarray}
u = \sum_{x\in \text{Crit }\omega}\frac{e^{i\frac{\pi}{4}(n_+-n_-)(x)}}{ \sqrt{\vert\det \omega^{\prime\prime}(x)\vert }}
\delta_{\omega(x)} +  r,
\end{eqnarray}
where 
$r \in \mathcal{D}^\prime(\mathbb{R})$
is an oscillatory
integral whose 
asymptotic symbol
has leading term
$\underset{x\in \text{Crit }\omega}{\sum}a(x)b_{-1}(x)\lambda^{-1}$
for 
$a(x)=\frac{e^{i\frac{\pi}{4}(n_+-n_-)(x)}}{ \sqrt{\vert\det \omega^{\prime\prime}(x)\vert }}$.
We rewrite the above formula in a simpler form:
\begin{equation}
u= \sum_{x\in \text{Crit }\omega} a(x)\delta_{\omega(x)}+ r
\end{equation}
where 
$a(x)=\frac{e^{i\frac{\pi}{4}(n_+-n_-)(x)}}{ \sqrt{\vert\det \omega^{\prime\prime}(x)\vert }}$.

Our goal is to express $\chi(S)$ in terms
of the coefficients 
$a(x),x\in\text{Crit }\omega$.
We recall the definition of
the Morse counting polynomial 
$\mathcal{M}_{\omega}(T)$ for a given 
Morse function $\omega$
(\cite[Definition C.4 p.~228]{Farber} ):
\begin{equation}
\mathcal{M}_{\omega}(T)=\sum_{x\in \text{Crit }\omega} T^{n_-(x)}.
\end{equation}
Observe that
\begin{eqnarray*}
n_-(x)&=&0\mod(2)\Leftrightarrow a(x)\in i\mathbb{R}\\
n_-(x)&=&1\mod(2)\Leftrightarrow a(x)\in \mathbb{R}\\
&\Rightarrow & -\frac{a(x)^2}{\vert a(x)\vert^2}=(-1)^{n_-(x)}.
\end{eqnarray*}
This implies that:
\begin{eqnarray}
\mathcal{M}_{\omega}(-1)=\sum_{x\in \text{Crit }\omega} \left(-1\right)^{n_-(x)}=
\sum_{x\in \text{Crit }\omega}-\frac{a(x)^2}{\vert a(x)\vert^2}.
\end{eqnarray}
To conclude, 
we use the well known result $\chi(S)=\mathcal{M}_{\omega}(-1)$ 
which is a consequence of the 
Morse inequalities
(\cite[Theorem C.3 p.~228]{Farber} and \cite[Thm 5.2 p.~29]{Milnor-Morse})
.
\end{proof}

Let $\mathcal{R}$ be the Radon transform.
From the identity relating
the Fourier transform and the Radon transform
(see \cite[p.~19]{WF1}),
we find a relationship between
the distribution $u$ of Theorem \ref{mainthmprecise}
and $\mathcal{R}\mu$:
\begin{eqnarray*}
\nonumber\mathcal{R}\mu(\omega,\tau)=\mathcal{F}^{-1}_\lambda\left(\widehat{\mu}(\lambda\omega)\right)(\tau)&\implies &  u(\tau)=\frac{1}{2i\pi}\partial_\tau\mathcal{R}\mu(\omega,\tau)\\
\implies   \frac{1}{2i\pi}\partial_\tau\mathcal{R}\mu(\omega,\tau)&=&\sum_{x\in \text{Crit }\omega} a(x)\delta_{\omega(x)} + r.
\end{eqnarray*}
This immediately proves a precised version of
Theorem \ref{radon} and gives
some informations on the Radon transform
of the measure $\mu$:
\begin{theorem}
Under the assumptions of Theorem
\ref{mainthmprecise}, let $\mathcal{R}\mu$ be the Radon transform
of $\mu$. Then for an everywhere
dense set of $\omega\in\mathbb{S}^2$, 
\begin{eqnarray}
\frac{1}{2i\pi}\partial_\tau\mathcal{R}\mu(\omega,\tau)=\sum_{x\in \text{Crit }\omega} a(x)\delta_{\omega(x)} + r
\end{eqnarray}
where $$\chi(S)=\sum_{x\in \text{Crit }\omega}-\frac{a(x)^2}{\vert a(x)\vert^2}.$$
\end{theorem}

\subsubsection{Remark.}
If we are given a Lagrangian distribution
$u\in\mathcal{D}^\prime(\mathbb{R})$ 
with singular point 
$t_0$ and
$u$ can be written as a sum
$u=a\delta_{t_0} + 
r$ 
where $a\in\mathbb{C}$ and 
$r$
is a Lagrangian distribution 
with asymptotic symbol
of degree $-1$,
then we can recover 
$a$ by scaling around
the point $t_0$.
Indeed, a straightforward
calculation yields
\begin{eqnarray}
\lim_{\lambda\rightarrow 0}\lambda u(\lambda(.-t_0) )= a\delta_{t_0}
\end{eqnarray}
where the limit
is understood in the sense
of distributions. For any
test function
$\varphi$ which is equal to $1$
in a sufficiently
small neighborhood
of $t_0$, 
we find
that 
$\lim_{\lambda\rightarrow 0}
\left\langle\lambda t(\lambda(.-t_0) ),\varphi\right\rangle=a$.
\subsection{The example of the sphere.}

We show how our proof
works in the case of the
unit sphere
$\mathbb{S}^2$
in $\mathbb{R}^3$.
First, note that 
any height function
restricted
on $\mathbb{S}^2$
is Morse 
excellent and has 
exactly
two critical points.
The Fourier transform
of $\mu$ is given by the exact
formula:
$\widehat{\mu}(\xi)
=4\pi\frac{\sin(\vert\xi\vert)}{\vert\xi\vert}$
therefore
\begin{eqnarray*}
\frac{\lambda}{2\pi}\widehat{\mu}(\lambda\omega)
&=& 2\sin(\lambda)\\
\implies \mathcal{F}^{-1}\left(\frac{\lambda}{2\pi}\widehat{\mu}(\lambda\omega) \right)
&=& \mathcal{F}^{-1}\left(2\sin(\lambda)\right)\\
=  \mathcal{F}^{-1}\left(\frac{e^{i\lambda}-e^{-i\lambda}}{i}\right)
&=& i(\delta_{+1}-\delta_{-1}).
\end{eqnarray*}
Finally, the identity \ref{Morsepol}
allows to recover the
well known result
$\chi(\mathbb{S}^2)=-i^2-(-i)^2=2$.

\section{Plane wave analysis and topology.}
Let us recall the statement of Theorem \ref{thm2}
before we give a proof:

\emph{Let $S$ be a closed
compact
surface embedded in
$\mathbb{R}^3$
and $\mu$ the associated 
surface carried measure.
If 
$\psi\in C^\infty(\mathbb{R}^3)$ 
is such that
its restriction on $S$ 
is Morse excellent then
we can recover
the Euler characteristic of $S$
from the map 
$\lambda\longmapsto\mu\left(e^{i\lambda \psi} \right)$.}

The idea to consider
oscillatory integrals
of the form $\mu\left(e^{i\lambda \psi} \right)$
comes from the
coordinate invariant definition
of wave front set
due to Gabor \cite{Gabor-72,WF2} then corrected 
by Duistermaat \cite[Proposition 1.3.2]{Duistermaat}.
To prove Theorem \ref{thm2},
we just repeat
the proof of Theorem \ref{mainthmprecise}
applied to the
distribution
$u=\mathcal{F}_\lambda^{-1}
\left(\left(\frac{\lambda}{2\pi}\right)
\mu\left(e^{i\lambda \psi} \right)\right)$
where
we replace
$\omega$ by $\psi$.

\section{The wave equation, propagation of singularities and
topology.}
Let us recall the statement of Theorem \ref{thm3}:

\emph{Let $S$ be a closed
compact 
surface embedded in
$\mathbb{R}^3$ and
$\mu$ the associated 
surface carried measure.
Consider the solution 
$u\in\mathcal{D}^\prime(\mathbb{R}^{3+1})$
of the wave equation 
$\left(\partial^2_t-\sum_{i=1}^3\partial_{x^i}^2\right) u=0$ 
with Cauchy data $u(0)=0,\partial_tu(0)=\mu$.
For $x$ in some
open dense subset of $\mathbb{R}^3$,
set $\ell(x)$ to be
the line $\{x\}\times \mathbb{R}\subset \mathbb{R}^{3+1} $,
then: 
\begin{enumerate}
\item the restriction
$u_{\ell(x)}$ is a compactly 
supported distribution of $t$, 
\item we can recover the 
Euler characteristic of
$S$ from the restriction 
$u_{\ell(x)}$.
\end{enumerate}}

Let us prove a precised version of 
claim (1):
\begin{prop}\label{claim1thm3}
Let $S$ be a closed
compact
surface embedded in 
$\mathbb{R}^3$
and $\mu$ the associated 
surface carried measure. 
Consider the solution $u$
of the wave equation $\square u=0$ with Cauchy data $u(0)=0,\partial_tu(0)=\mu$.
For all $x\in\mathbb{R}^3\setminus S$:
\begin{itemize}
\item $t\longmapsto u(t,x)$ is a compactly supported distribution of $t$ 
\item the singular support of $u(.,x)$ is
a subset of 
$\{t\in\mathbb{R}| \exists y\in S \text{ s.t. } \vert x-y\vert=\vert t\vert, y-x \perp T_{y}S \}$.
\end{itemize}
\end{prop}
\begin{proof}

Recall that $\ell_x=\{(t,x)| t\in\mathbb{R}^3   \} $ and denote
by $N^*(\ell_x)\subset T^*\mathbb{R}^{3+1}$ its
conormal bundle. We want to prove that one can restrict
the distribution $u$ on $\ell_x$. 
First, $\square u=0$ implies that
the wave front set of $u$ 
lies in the 
characteristic set of $\square$ (\cite[Theorem 8.3.1]{HormanderI}):
\begin{eqnarray*}
WF(u)\subset \text{Char }\square= \{\tau^2-\vert\xi\vert^2\} \implies 
WF(u)\cap N^*(\ell_x)=\emptyset
\end{eqnarray*}
which means that one can pull--back the distribution
$u$ by the embedding $i:\ell_x\hookrightarrow \mathbb{R}^{3+1}$ 
(see \cite[Theorem 8.2.4]{HormanderI}). The restriction
$i^*u$ is thus well defined, it is
compactly since the Cauchy data $(0,\mu)$
is compactly supported
and by finite propagation 
speed property for the wave equation.

Secondly, we calculate the wave front set
of the restriction $i^*u$. Elements of the
cotangent space $T^*\mathbb{R}^{3+1}$
are denoted by $(t,x;\tau,\xi)$
and $\pi$ is the projection 
$T^*\mathbb{R}^{3+1}\mapsto \mathbb{R}^{3+1}$.
Since $u$ is solution of 
$\square u=0$ with Cauchy data $(0,\mu)$, 
$u$ is given by the representation
formula (\cite[equation (2.21) p.~12]{BersJohnSchechter},
\cite[Theorem 5.3 p.~ 67]{AlinhacHyperbolicPDE},
\cite[p.~3]{SoggeNonlinearWave})
\begin{eqnarray}\label{intformula}
u(t,x)=\frac{1}{4\pi t}\int_{\mathbb{R}^3} \delta(\vert t\vert-\vert x-y\vert)\mu(y)dy. 
\end{eqnarray}
We want to calculate
$WF(u)$ using the integral
formula (\ref{intformula}).
By finite propagation 
speed,
the condition $x\in\mathbb{R}^3\setminus S$
ensures that $i^*u=0$ in some neighborhood
of $t=0$ which means we do not have
to consider the contribution of $t=0$
to $WF(\delta(\vert t\vert-\vert x-y\vert))$.
Denote by $P$ the projection $P:(t,y)\in
\mathbb{R}^{3+1}\mapsto t\in \mathbb{R}$,
since
$WF(\delta(\vert t\vert-\vert x\vert))=\{(t,x;\lambda t,-\lambda x) 
| \vert t\vert=\vert x\vert,
\lambda\in\mathbb{R}\setminus \{0\} \}
\cup T_0^*\mathbb{R}^{3+1}$,
the calculus of wave front set
yields:
\begin{eqnarray*}
WF(u)&\subset & P_*WF(\delta(\vert t\vert-\vert x-y\vert)\mu(y))\\
&\subset & \{(t;\lambda t)| \exists (y;\eta)\in N^*(S), \vert t\vert=\vert x-y\vert, \lambda(x-y)=\eta \}\\
&=&\{(\pm \vert x-y\vert;\tau) | (x-y) \perp T_yS ,\tau\in\mathbb{R}\setminus \{0\} \}\\
\implies  \text{ss }u=\pi(WF(i^*u))&\subset & \{ \pm\vert x-y\vert | (x-y)\perp T_yS  \}.
\end{eqnarray*}
\end{proof}

Let us prove that one can recover
$\chi(S)$ from $u(.,x)\in\mathcal{D}^\prime(\mathbb{R})$
concluding the proof
of Theorem \ref{thm3}.

\begin{theorem}\label{thm3precise}
Let $S$ be a closed
compact 
surface embedded in
$\mathbb{R}^3$ and
$\mu$ the associated 
surface carried measure.
Consider the solution 
$u\in\mathcal{D}^\prime(\mathbb{R}^{3+1})$
of the wave equation 
$\left(\partial^2_t-\sum_{i=1}^3\partial_{x^i}^2\right) u=0$ 
with Cauchy data $u(0)=0,\partial_tu(0)=\mu$.
Then for $x$ in some
open dense subset of $\mathbb{R}^3$, we
have the canonical decomposition:
\begin{equation}
-2i\left(t\partial_t+1 \right)u(t,x)
=\sum_{y\in \text{Crit }L_x} a(y)\delta_{\vert
y-x\vert}(t) + r(t)
\end{equation}
where $L_x:y\in S\mapsto \vert y-x\vert$ is Morse excellent,
$r$ is a finite
sum of oscillatory integrals with symbol
of degree $-1$.
\begin{equation}
\chi(S)=\sum_{y\in \text{Crit }L_x} -\frac{a(y)^2}{\vert a(y)\vert^2}.
\end{equation}
\end{theorem}

\begin{proof}
We again use the representation formula
for $u$:
\begin{eqnarray*}
u(t,x)=
\frac{1}{4\pi t}\int_{\mathbb{R}^3}  
\delta(\vert t\vert-\vert x-y\vert)\mu(y)dy.
\end{eqnarray*}
Recall that $x\in \mathbb{R}^3\setminus S$
implies $u(.,x)=0$ in some neighborhood 
of $t=0$.
Therefore, for $t\geqslant 0$:
\begin{eqnarray*}
u(t,x)&=&\frac{1}{8\pi^2 t} \int_{\mathbb{R}} d\lambda e^{it\lambda} 
\int_{\mathbb{R}^3} \mu(y)e^{-i\lambda\vert x-y\vert}dy \\
\implies
-2i\left(t\partial_t+1 \right)u(t,x) &= & \mathcal{F}^{-1}\left(\frac{\lambda}{2\pi}
\int_{\mathbb{R}^3} \mu(y)e^{-i\lambda\vert x-y\vert}dy
\right).
\end{eqnarray*}

By Lemma \ref{morseexcdist}, for $x$
in some open 
dense set in $\mathbb{R}^3$, 
the function $L_x:y\in S\mapsto \vert y-x\vert$
is 
Morse
excellent.
Therefore, repeating
the proof
of Theorem \ref{mainthm}
with $L_x$ instead of $-\omega$,
we find that
$$-2i\left(t\partial_t+1 \right)u(t,x)
=\sum_{y\in \text{Crit }L_x} a(y)\delta_{\vert
y-x\vert}(t) + r(t)$$
where $r$ is a finite
sum of oscillatory integrals with symbol
of degree $-1$.
Finally $\chi(S)=\sum_{y\in \text{Crit }L_x} -\frac{a(y)^2}{\vert a(y)\vert^2}$.
\end{proof}

\section{Appendix.}
In this section,
we gather several
important results
in Morse theory.

\subsection{Almost all height functions are Morse.}
Let $S$ be an 
embedded surface in 
$\mathbb{R}^3$.
For every $x\in S$,
we denote by $T_xS$ the 
tangent plane to $S$ at $x$.
\begin{defi}
The map 
$x\in S \mapsto n(x)\in\mathbb{S}^2$ where 
$n(x)$ is the
oriented unit normal
vector
to $T_xS$ 
is called the Gauss map. 
It induces canonically
a \emph{projective} 
Gauss map
denoted by 
$[n]:=x\in S\longmapsto [n](x)
\in\mathbb{RP}^2$.
\end{defi}
The next theorem \cite[Thm 11.2.2 p.~94]{DubrovinFomenkoNovikovII}
characterizes all height functions
which are Morse functions 
in terms of the Gauss map:
\begin{theorem}\label{regvalGaussmorse}
Let $S$ be an 
embedded surface in 
$\mathbb{R}^3$. 
The height
function $\xi\in C^\infty(S):=x\in S\mapsto \xi(x)$ 
is a Morse function precisely
when 
$[\xi]\in\mathbb{RP}^2$
is a regular value of the projective Gauss map $[n]$.
It follows
that for an \emph{open everywhere dense set} 
of $[\xi]\in\mathbb{RP}^2$, 
the height
function $\xi\in C^\infty(S)$ 
is a Morse function.
\end{theorem}

\subsection{Almost all Morse
height functions are excellent Morse functions.}
We
refine Theorem \ref{regvalGaussmorse} and show that
for generic $\xi\in\mathbb{RP}^2$,
the height function
$\xi$ is an excellent Morse function
i.e. all critical values
of $\xi$ are distinct.
\begin{lemm}\label{regvalGaussmorse2}
Let $S$ be a compact
embedded
surface in 
$\mathbb{R}^3$,
for an \emph{open everywhere dense set} 
of $[\xi]\in\mathbb{RP}^2$, 
the height
function 
$\xi\in C^\infty(S)$ 
is an excellent Morse function.
\end{lemm}
\begin{proof}
Let $V$ be the set of regular values
of the projective Gauss map $[n]$
in $\mathbb{RP}^2$. By Sard's Theorem
$V$ is open dense in $\mathbb{RP}^2$
and
\begin{eqnarray*}
[\xi]\in V\Leftrightarrow \text{ the height function }\xi \text{ is Morse}.
\end{eqnarray*}

We give next a characterization
of Morse height functions which
are not Morse excellent.
For all $[\xi]\in V$, there is a neighborhood
$\Omega$ of $[\xi]$ such that 
the preimage $[n]^{-1}(\Omega)$ 
is a disjoint union of 
open sets
$(U_1,\cdots,U_k)\subset S^k$ and
each $U_i$ is sent diffeomorphically
to $\Omega$ by the Gauss map.
Therefore there is a collection of maps
\begin{eqnarray*}
(x_1,\cdots,x_k):
=[\xi]\in\Omega\subset \mathbb{RP}^2\longmapsto (x_1([\xi]),\cdots,x_k([\xi]))\in S^k
\end{eqnarray*}
such that $\forall i, [n](x_i([\xi]))=[\xi]$.
We claim that for $[\xi]\in\Omega$:
\begin{eqnarray*}
\xi \text{ is not Morse excellent }
&\Leftrightarrow & \xi.\left(x_i([\xi])-x_j([\xi])\right)=0 \text{ for some }1\leq i<j\leq k.
\end{eqnarray*}
From the fact that
$$d_\xi\xi.\left(x_i([\xi])-x_j([\xi])\right)=\left\langle \left(x_i([\xi])-x_j([\xi])\right) , . \right\rangle 
\neq 0 $$
we deduce that the set of $[\xi]$ such that
the height function
$\xi$ is not Morse excellent 
is a finite union of submanifolds in $\Omega$.
It has thus \textbf{empty interior} 
which proves that Morse
excellent $\xi$ are open dense.
\end{proof}

\subsection{The distance function to almost every 
point
is Morse.}
The height function
is not the only
way to produce
Morse functions on
$S$. Recall we considered
the distance function to $x$,
$L_x:=y\in S\longmapsto 
\vert x-y \vert$.
The set of points $x\in\mathbb{R}^3$ 
where 
$L_x^2\in C^\infty(S)$ 
fails to be a Morse
function is called
the set of \emph{focal points}. 
By the result 
in \cite[Section 11.3 p.~95]{DubrovinFomenkoNovikovII}
and \cite[Corollary 6.2 p.~33]{Milnor-Morse}, the set
of \emph{focal
points} has 
null measure in
$\mathbb{R}^3$.
If 
$\psi\geqslant 0$ 
is a non negative 
Morse function
on $S$ with 
\textbf{only positive
critical values} then
\begin{eqnarray*}
d_x\psi(c)=0\implies d^2_x\sqrt{\psi}(c)=
\frac{d^2_x\psi(c)}{2\sqrt{\psi}} \neq 0. 
\end{eqnarray*}
This shows that 
if 
$x\notin S$ then
$L_x^2$
is a Morse function iff
$L_x$
is Morse. Thus:
\begin{prop}
For an \emph{open everywhere dense set}
of points $x\in\mathbb{R}^3$,
the distance
function 
\begin{eqnarray*}
L_x:=y\in S\longmapsto \vert x-y\vert
\end{eqnarray*}
is Morse.
\end{prop}

\subsection{The distance function to almost every 
point
is Morse excellent.}
The next Lemma is needed for the proof of Theorem
\ref{thm3}.
\begin{lemm}\label{morseexcdist}
Let $S$ be a closed compact
embedded surface 
in 
$\mathbb{R}^3$. 
For generic $x\in\mathbb{R}^3$, 
the 
function $L_x\in C^\infty(S)$ 
is an excellent Morse function.
\end{lemm}
\begin{proof}
Note that if $x\notin S$ and $L_x^2$
is Morse 
excellent
then so is $L_x$.
Therefore, it suffices
to prove that
$L_x^2$ is Morse
excellent for generic $x$.
For every 
$(y,x)\in S\times 
\mathbb{R}^3$, let use
denote by $g_x(y)$
the function $L_x(y)^2$.

For given $x_0$ s.t. $g_{x_0}$
is Morse,
we show there is a neighborhood
$U$ of $x_0$ such that
the critical points
of $L_x$ 
depend smoothly on $x\in U$.
Let us call $y_1(x_0),\cdots,y_k(x_0)$
the isolated critical points
of $g_{x_0}$.
For all $l\in\{1,\cdots,k\}$, 
$y_l(x_0)$
is a non degenerate
critical point
of $g_x$ i.e. the Hessian 
$d_y^2g_x$
is invertible. 
Therefore,
we can use 
the implicit
function
theorem 
to express the critical
points
$(y_i)_{1\leq i\leq k}$
of $g_x$ as functions
of $x\in U$
in such a way that
\begin{eqnarray}
\forall x\in U, \forall i\in\{1,\cdots, k \}, d_y g_x(y_i(x))=0.
\end{eqnarray}
Hence 
\begin{eqnarray*}
L_x^2\text{ not Morse excellent }\Leftrightarrow
g_x(y_i(x))=g_x(y_j(x)) \text{ for some } 1\leq i<j\leq k.
\end{eqnarray*}
Observe that the subset 
$\Sigma_{ij}=\{ x\in U | g_x(y_i(x))=g_x(y_j(x))  \}$
is a surface
in $U$ since
$$\partial_x g_x(y_i(x))
-\partial_xg_x(y_j(x))
=2\left\langle . , y_i(x)-y_j(x)\right\rangle\neq 0 ,$$
therefore the set of $x$
such that $L_x^2$ fails to be
Morse excellent has empty interior
in $U$ which proves the claim.
\end{proof}

\bibliographystyle{amsplain}
\bibliography{qed}

\providecommand{\bysame}{\leavevmode\hbox to3em{\hrulefill}\thinspace}
\providecommand{\MR}{\relax\ifhmode\unskip\space\fi MR }
\providecommand{\MRhref}[2]{%
  \href{http://www.ams.org/mathscinet-getitem?mr=#1}{#2}
}
\providecommand{\href}[2]{#2}
\begin{thebibliography}{10}

\bibitem{Aizenbud-12}
A.~Aizenbud and V.~Drinfeld, \emph{The wave front set of the {F}ourier
  transform of algebraic measures},  (2014), arXiv:1212.3630.

\bibitem{AlinhacHyperbolicPDE}
S.~Alinhac, \emph{Hyperbolic partial differential equations}, Universitext,
  Springer, 2009.

\bibitem{BersJohnSchechter}
L.~Bers, F.~John, and M.~Schechter, \emph{Partial differential equations},
  third ed., American Mathematical Society, 1971.

\bibitem{WF2}
Christian Brouder, Nguyen~Viet Dang, and Fr\'ed\'eric H\'elein,
  \emph{Boundedness and continuity of the fundamental operations on
  distributions having a specified wave front set.}, in preparation.

\bibitem{WF1}
\bysame, \emph{A smooth introduction to the wavefront set}, arXiv:1404.1778.

\bibitem{SoggeNonlinearWave}
Christopher D.Sogge, \emph{Lectures on non--linear wave equations}, second ed.,
  International Press, 2008.

\bibitem{DubrovinFomenkoNovikovII}
B.A. Dubrovin, A.T. Fomenko, and S.P. Novikov, \emph{Modern {G}eometry
  {II}--{M}ethods and {A}pplications}, Springer, 1985.

\bibitem{Duistermaat}
J.~J. Duistermaat, \emph{Fourier integral operators}, Birkh{\"{a}}user, Boston,
  1996.

\bibitem{Eskin}
G.~Eskin, \emph{Lectures on {L}inear {P}artial {D}ifferential {E}quations},
  Graduate Studies in Mathematics, vol. 123, Amer. Math. Soc., Providence,
  2011.

\bibitem{Farber}
M.~Farber, \emph{Topology of closed one-forms}, Mathematical surveys and
  monographs, vol. 108, AMS, 2004.

\bibitem{Fuindex}
J.~Fu, \emph{Curvature measures of subanalytic sets}, Amer. J. Math.
  \textbf{116} (1994), 819--880.

\bibitem{Gabor-72}
A.~Gabor, \emph{Remarks on the wave front of a distribution}, Trans. Am. Math.
  Soc. \textbf{170} (1972), 239--44.

\bibitem{HormanderI}
L.~H{\"o}rmander, \emph{The {A}nalysis of {L}inear {P}artial {D}ifferential
  {O}perators {I}. {D}istribution {T}heory and {F}ourier analysis}, second ed.,
  Springer Verlag, Berlin, 1990.

\bibitem{KDB}
J.L.Brylinski, A.~Dubson, and M.~Kashiwara, \emph{Formule de l'indice pour
  modules holon\^omes et obstruction d'{E}uler locale}, C. R. Acad. Sci. Paris
  S\'er. I \textbf{293} (1981), 573--576.

\bibitem{Kashiwaraindex}
M.~Kashiwara, \emph{Index theorem for constructible sheaves}, Ast\'erisque
  \textbf{130} (1985), 193--209.

\bibitem{Milnor-Morse}
J.~Milnor, \emph{Morse theory}, Princeton University Press, Princeton, 1969.

\bibitem{ReedSimonII}
M.~Reed and B.~Simon, \emph{Methods of modern mathematical physics. {II}
  {F}ourier analysis, self-adjointness}, Academic Press, New York, 1975.

\bibitem{BatesWeinstein}
S.Bates and A.Weinstein, \emph{Lectures on the geometry of quantization},
  {B}erkeley mathematics lecture notes ed., vol.~8, AMS, 1998.

\bibitem{NovikovTaimanov}
S.P.Novikov and I.A.Taimanov, \emph{Modern geometric structures and fields},
  Graduate Studies in Mathematics, vol.~71, AMS, 2006.

\bibitem{SteinBeijing}
Elias~M Stein, \emph{Beijing lectures in harmonic analysis}, Annals of
  Mathematics Studies, Princeton University Press, 1986.

\bibitem{SteinShakarchi}
Elias~M Stein and Rami Shakarchi, \emph{Functional analysis, an introduction to
  further topics in analysis}, Princeton University Press, 2011.

\end{thebibliography}

\end{document}